\documentclass[12pt]{article}
\usepackage[english]{babel}
\usepackage{sfmath}

\usepackage{amssymb}
\usepackage{natbib}

\usepackage[leqno]{mathtools}

\usepackage{amsthm}

\usepackage{mathrsfs}
\usepackage{xr-hyper}
\usepackage[bookmarks,pdfnewwindow,plainpages=false]{hyperref}

\hypersetup{
implicit=false,
colorlinks=true,
linkcolor=black,
citecolor=black,
pdfauthor={Victor Ivrii},
pdftitle={Sharp Spectral Asymptotics for Dirac Energy},
pdfsubject={Sharp Spectral Asymptotics},
pdfkeywords={Microlocal Analysis,  Schr\"odinger Operator},
}

\newcommand\loc{{\rm {loc}}}

\usepackage{subfig}

\usepackage{enumitem}

\newcommand{\3}{{|\!|\!|}}

\newcommand\W{{\rm W}}

\newcommand\Def{{\overset {\rm {def}}{\ =\ }}}

\newcommand\dist{{\rm {dist}}}
\newcommand\eff{{\rm {eff}}}

\newcommand\boldalpha{{\boldsymbol \kappa }}
\newcommand\boldbeta {{\boldsymbol \nu }}
\newcommand\boldtau{{\boldsymbol \tau}}

\newcommand\Hess{\operatorname{Hess}}

\newcommand\Ker{\operatorname{Ker}}

\newcommand\mes{\operatorname{mes}}

\newcommand\rank{\operatorname{rank}}

\newcommand\supp{\operatorname{supp}}
\newcommand\tr{\operatorname{tr}}
\newcommand\Tr{\operatorname{Tr}}

\newcommand\bR{{\mathbb R}}

\newcommand\cI{{\mathcal I}}
\newcommand\cJ{{\mathcal J}}

\newcommand\cU{{\mathcal U}}

\newtheorem{theorem}{Theorem}[section]

\newtheorem{corollary}[theorem]{Corollary}
\newtheorem{proposition}[theorem]{Proposition}

\theoremstyle{definition}
\newtheorem{definition}[theorem]{Definition}

\theoremstyle{remark}
\newtheorem{remark}[theorem]{Remark}

\numberwithin{equation}{section}

\newenvironment{claim}[1][{\rm(\theequation)}]{\refstepcounter{equation}%
\begin{trivlist}
\item[{\hskip\labelsep#1}]}{\end{trivlist}\addvspace{10pt}}

\newcounter{note}

\newenvironment{claim*}[1]{\medskip
\begin{trivlist}
\item[{\hskip\labelsep#1}]}{\medskip\end{trivlist}}

\newenvironment{phantomequation}[1][]{\refstepcounter{equation}}{}

\renewcommand\subsubsection{\paragraph{\thesubsubsection}\refstepcounter{subsubsection}}

\begin{document}


\title{%
Sharp Spectral Asymptotics for Dirac Energy}
\author{Victor Ivrii}
\date{\today}

\maketitle

{\abstract%
I derive sharp semiclassical asymptotics  of $\int |e_h(x,y,0)|^2\omega(x,y)\,dx\,dy$ where $e_h(x,y,\tau)$ is the Schwartz kernel of the spectral projector and $\omega(x,y)$ is singular as $x=y$. I also consider asymptotics of more general expressions.
\endabstract}

\section{Introduction}

In the series of papers \cite{IS, MQT1,MQT2,MQT3} devoted to the Sharp Asymptotics of the Ground State Energy of Heavy Atoms and Molecules it was needed to calculate \emph{Dirac Correction Term\/}\footnote{\label{foot-1} Representing Coulomb interaction of electrons with themselves which should not to be counted in the energy calculation and should be subtracted from the Thomas-Fermi expression.}  which in that approximation was equal to
\begin{equation}
I=\Def\iint |e(x,y,\tau)|^2|x-y|^{-1}\,dx\,dy
\label{0-1}
\end{equation}
where $e(x,y,\tau)$ is the Schwartz kernel of the spectral projector $E(\tau)$ of the (magnetic) Schr\"odinger operator 
\begin{equation}
A= {\frac 1 2}\Bigl(\sum_{j,k}P_jg^{jk}(x)P_k -V\Bigr),\qquad P_j=h D_j-\mu V_j
\label{0-2}
\end{equation}
$\tau\approx 0$ and $h\to +0$ (while either  $\mu\to +\infty$ or remains constant). Actually the corresponding part of these papers was originally more complicated but it was reduced to the  problem above.

Then $I\asymp h^{-d-1}$ where $d$ is the dimension ($d=3$ in the above papers) and it was needed to prove that $I=\cI+ O(h^{-d-1+\delta})$ with $\cI$ defined by the same formula but with $e(x,y,\tau)$ replaced by 
\begin{equation}
e^\W_y (x,y,\tau) \Def (2\pi h)^{-d}\int_{g(y,\xi)\le V(y)+2\tau} e^{ih^{-1}\langle x-y,\xi\rangle }\,d\xi
\label{0-3}
\end{equation}
and   with a small exponent $\delta>0$; for Magnetic Schr\"odinger it was needed to prove as $\mu \le h^{ -\delta}$ only. Expression (\ref{0-3}) is a \emph{Weyl expression\/} for $e(x,y,\tau)$ for operator with coefficients frozen at point $y$.

However I believe that the asymptotics of expression (\ref{0-1}) or more general one is interesting by itself and that there are a sharp asymptotics. Still my attempts to derive it were not very successful and in \cite{EE} I made some claims which I could not sustain at that time. So in this paper I just want to bring some degree of the order to this matter.

I am going to consider a matrix $h$-differential operator $A(x,hD)$ and find asymptotics of 
\begin{equation}
I\Def\iint  \omega  (x,y) e(x,y,\tau)\psi_2(x) e (y,x,\tau) \psi_1(y)\,dx\,dy
\label{0-4}
\end{equation}
with a matrix-valued function $\omega(x,y)$ such that
\begin{claim}\label{0-5}
$\omega (x,y)\Def \Omega
(x,y; x-y)$ where function $\Omega$ is smooth in 
$B(0,1)\times B(0,1)\times B(\bR^d\setminus 0)$ and homogeneous of degree $-\kappa$ ($0<\kappa<d$) with respect to its third argument\footnote{\label{foot-2} In other words it is Michlin-Calderon-Zygmund kernel.} 
\end{claim}
and  with smooth cut-off functions $\psi_1,\psi_2$.

The main part of asymptotics should have a magnitude of $h^{-d-\kappa}$ and I would like to get a remainder estimate $O(h^{1-d-\kappa})$.

One can also consider a more general expression
\begin{multline}
I_m\Def\iint  \omega  (x^1,\dots,x^m) e(x^1,x^2,\tau)\psi_2(x^1) e (x^2,x^3,\tau) \cdots e(x^m,x^1,\tau) \psi_{m+1}(x^0)\times\\
dx^1\cdots dx^m
\label{0-6}
\end{multline}
with $x^{m+1}=x^1$, $\psi_{m+1}\Def \psi_1$ etc and 
\begin{claim}\label{0-7}
$\omega (x^1,\dots,x^m)\Def \Omega  (x^1,\dots,x^m; \{x^j-x^{j+1}\}_{1\le j \le m })$ where function $\Omega$ is smooth in 
$B(0,1)^m \times B(\bR^d\setminus 0)^{m-1}$ and homogeneous of degree $-(m-1)\kappa$ with respect to  $ \{x^j-x^k\}_{1\le j <k\le m}$. Moreover, 
\begin{equation*}
|D^{\boldbeta}_{\mathbf z} D^{\boldalpha}_{\mathbf x} \Omega|\le C_{{\boldbeta},{\boldalpha}}|z^1|^{-\kappa -|\beta^1|}\cdots |z^m|^{-\kappa -|\beta^m|}\qquad
\text{as } \sum_k |z^k|^2=1,\sum_k z^k=0
\end{equation*}
where ${\mathbf x}=(x^1,\dots,x^m)$, ${\mathbf z}=(z^1,\dots,z^m)$, etc.
\end{claim}
However I will leave it for another paper since not of all my arguments I was able to implement in this case.

The main part of asymptotics should have a magnitude of $h^{-d-(m-1)\kappa}$ (see Theorem \ref{thm-1-6})  and I would like to get a remainder estimate $O(h^{1-d-(m-1)\kappa})$.

I am also leaving for another paper  the similar but much more delicate and difficult analysis for a 2-dimensional Magnetic Schr\"odinger operator(\ref{0-2}) with the trajectories having many loops.

\begin{remark}\label{rem-0-1}
(i) To avoid the necessity to cut-off with respect to $hD$ one needs to assume that its symbol satisfies
\begin{equation}
|a (x,\xi)|^{-1}\le C|\xi|^{-m}\qquad \text{as\ } |\xi|\ge C_0
\label{0-8}
\end{equation}
as $a\in \Psi^m$ (one can weaken this condition but I leave it to the reader);

\medskip
\noindent
(ii) One needs to assume that $a$ is semibounded from below which under (\ref{0-8}) is equivalent to
\begin{equation}
\langle a(x,\xi)v,v\rangle \ge c^{-1}|v|^2  \qquad \text{as\ } |\xi|\ge C_0;
\label{0-9}
\end{equation}
otherwise instead of $E(\tau)$ one should consider $E(\tau_1,\tau_2)\Def E(\tau_2)-E(\tau_1)$; I leave it to the reader as well.
\end{remark}

This paper consist of two sections: in   section \ref{sect-1} I derive asymptotics with the sharp remainder estimate but with the implicit Tauberian approximation for $e(x,y,0)$.  In section \ref{sect-2} is I replace it by   expression (\ref{0-3}) without deteriorating remainder estimate for scalar operators under mild non-degeneracy condition (theorem \ref{thm-2-19}) and for certain matrix operators  (theorem \ref{thm-2-20}(i)) and with some not sharp remainder estimates for other matrix operators (theorem \ref{thm-2-20}(ii)). I just mention that for larger $\kappa$ we need less restrictive conditions to operator.

\section{Estimates}\label{sect-1}

\subsection{Special case}
\label{sect-1-1}

Let us assume first that $\omega=1$ but relax conditions to $\psi_1,\dots,\psi_m$, assuming only that $\psi_1,\dots,\psi_m\in L^\infty$. This is definitely not the case I am interested in but one needs to make few clarifications first. Then 
\begin{equation}
I_m\Def\Tr E(\tau) \psi_2 E(\tau)\psi_3 E_(\tau)\cdots E(\tau)\psi_{m+1}
\label{1-1}
\end{equation}
containing $m$ factors $E(\tau)$.

Under condition (\ref{0-9}) it is known (see f.e. \cite{Ivr1}) that if $L^\infty$ norms and the diameters of supports $\psi$, $\psi_1$ are bounded, then 
\begin{equation}
\3\psi E(\tau) \psi_1\3_1 \le Ch^{-d}\qquad\text{as } |\tau|\le c
\label{1-2}
\end{equation}
where $\3.\3_\infty$ and $\3.\3_1$ denote  operator and trace norms respectively. Then   since an operator norm of $E(\tau)$ does not exceed $1$ I conclude that $|I_m|\le ch^{-d}$. So
\begin{claim}\label{1-3}
If $\psi_j\in L^\infty$ and $I_m$ is given by (\ref{1-1}) then $|I_m|\le Ch^{-d}$.
\end{claim}

Further, let us assume that
\begin{claim}\label{1-4}
$a(x,\xi)$ is microhyperbolic on energy level $0$. 
\end{claim}
Then as (\ref{1-4}) is fulfilled on supports of $\psi$, $\psi_1$  it is known (see f.e. \cite{Ivr1}) that 
\begin{equation}
\3 \psi \bigl(E(\tau) -E(\tau')\bigr)\psi_1 \3 \le C(|\tau-\tau'|+hT^{-1})h^{-d}
\qquad\text{as }|\tau|\le \epsilon_1, \ |\tau'|\le c.
\label{1-5}
\end{equation}
Here and for a while $T\asymp1$ but I want to keep a track of it.

Since this property holds under wider assumptions than microhyperbolicity, I will assume so far only that (\ref{1-5}) holds.

Then 
\begin{equation}
|\Tr'  \Bigl(\bigl(E(\tau)-E(\tau')\bigr)\psi_2 E(\tau_2)\psi_3 E(\tau_3)\cdots E(\tau_m)\psi_{m+1} \Bigr)|
\label{1-6}
\end{equation}
also does not exceed the right hand expression of (\ref{1-5}) as $|\tau|\le\epsilon_1$ and therefore due to the standard Tauberian arguments (second part, see f.e.\cite{Ivr1}) the following inequality holds:
\begin{multline}
|\Tr'  \biggl(\Bigl(E(0)-h^{-1}\int _{-\infty}^0 F_{t\to h^{-1}\tau} \bigl({\bar\chi}_T(t) U(t)\bigr)\,d\tau \Bigr) \psi_2 E(\tau_2)\psi_3 E(\tau_3)\cdots E(\tau_m)\psi_{m+1} \biggr)|\le \\
CT^{-1}h^{1-d}
\label{1-7}
\end{multline}
where I use my  standard notations ${\bar\chi}$ and $\chi$ in the future and ${\bar\chi}(t)={\bar\chi}(t/T)$ etc (see f.e. \cite{IRO1}). Here  and below $\Tr'$  is the ``scalar trace'' of the operator, and does not include taking matrix trace $\tr$.

Here and below $U(t)=e^{ih^{-1}tA}$ is the propagator of $A$ and $u(x,y,t)$ is its Schwartz' kernel.

So with $O(T^{-1}h^{1-d})$ error one could replace  one copy of $E(0)$ in $I_m$ by its standard implicit Tauberian approximation
\begin{equation}
h^{-1}\int _{-\infty}^0 F_{t\to h^{-1}\tau} \bigl({\bar\chi}_T(t) U(t)\bigr)\,d\tau
\label{1-8}
\end{equation}
and in by the virtues of the same arguments I can do it with another copy of $E(0)$. Therefore

\begin{proposition}\label{prop-1-1}
Under conditions $(\ref{1-5})$ with an error $O(T^{-1}h^{1-d})$ $I_m$ is equal to
\begin{equation}
h^{-m}\Tr'\int _{{\boldtau}\in \bR^{-,m}}
F_{{\mathbf t}\to  h^{-1}{\boldtau} } \Bigl({\bar\chi}_T(t_1) U(t_1)  \psi_2 {\bar\chi}_T(t_2) U(t_2)\psi_3\cdots U(t_m)\psi_{m+1} )\Bigr)  \,d\boldtau
\label{1-9}
\end{equation}
with ${\mathbf t}=(t_1,\dots,t_m)$, ${\boldtau}=(\tau_1,\dots,\tau_m)$.
\end{proposition}

Note that here one can take any $T\in [Ch^{1-\delta}, c]$ (but then an error depends on $T$). Further, note that as $\dist(\supp \psi_j,\supp \psi_{j+1})\ge (c_0+\epsilon)T$ where $c_0$ here and below is the upper bound of the propagation speed on energy level $0$ and $x^{m+1}\Def x^1$, expression (\ref{1-9}) as $m=2$ or the similar expression as $m\ge 3$ become negligible  and I arrive to

\begin{corollary}\label{cor-1-2}
If in frames of proposition \ref{prop-1-1}  $\dist(\supp \psi_j,\supp \psi_{j+1})\ge (c_0+\epsilon) T$ for some $j=1,\dots,m$  then $|I_m|$ does not exceed $CT^{-1}h^{1-d}$.
\end{corollary}

\subsection{Smooth case} 
\label{sect-1-2}

The next step is to assume that $\omega$ is a  smooth function. Without any loss of the generality one can assume that $\omega$ is also compactly supported (since $\psi,\psi_1$ are). Then from
\begin{multline}
\omega (x^1,\dots,x^m) = \int \omega (y^1,\dots,y^m)\delta (y^1-x^1, \dots, y^m-x^m)\,dy =\\  \int \omega' (y^1,\dots,y^m) 
\theta (y^1-x^1)\cdots \theta (y^m-x^m)\,dy^1\cdots dy^m
\label{1-10}
\end{multline}
one arrives to 
\begin{equation}
I_m =\int\omega' (y^1,\dots,y^m) J_2(y^1,\dots,y^m)\,dy^1\cdots dy^m
\label{1-11}
\end{equation}
with $J_2(y^1,\dots,y^m)$ defined by $\omega=1$ and $\psi_j (x)$ redefined as $\psi_j(x)\theta (y^j-x)$ where here and below  $\theta (x)=\theta (x_1)\cdots \theta (x_d)$. Then I immediately arrive to

\begin{proposition}\label{prop-1-3} Let $\omega$ and $\psi_1,\dots,\psi_m$  be   smooth functions and let condition $(\ref{0-9})$ be fulfilled. Then $|I_m|\le Ch^{-d}$.
\end{proposition}

\begin{remark}\label{rem-1-4} As $m=2$ and $\omega,\psi_1,\psi_2\in L^\infty$
$|I_2|\le Ch^{-d}$ obviously (it follows from the estimate 
$\3 \psi  E\psi\3_2\le Ch^{-d/2}$ where $\3.\3_2$ is the Hilbert-Schmidt norm).
Can one prove the similar result for $m\ge 3$?
\end{remark}

\begin{proposition}\label{prop-1-5}  Let $\omega$ and $\psi_1,\dots, \psi_m$  be   smooth functions and let conditions  $(\ref{0-9})$  and $(\ref{1-5})$ be fulfilled. Then 

\medskip
\noindent
{\rm (i)} with an error $O(T^{-1}h^{1-d})$ $I_m$ is equal to
\begin{multline}
\cI_m= h^{-m}\int \int _{\boldtau\in \bR^{-,m}}\omega (x^1,\dots,x^m)
F_{{\mathbf t}\to  h^{-1}\boldtau }  
\Bigl({\bar\chi}_T(t_1) u(x^1,x^2,t_1)\psi_2 (x^2) {\bar\chi}_T(t_2)  \times\\ u(x^2,x^3, t_2)\psi_3(x^3) \cdots
U(t_m)\psi_{m+1} (x^{m+1})\Bigr)  \,d\tau\, dx^1\cdots dx^m
\label{1-12}
\end{multline}
with $x^{m+1}\Def x^1$.

\medskip
\noindent
{\rm (ii)} Further, if  $\dist(\supp \psi_j,\supp \psi_{j+1})\ge (c_0+\epsilon)T$ for some $j=1,\dots,m$  then  $|I_m|$ does not exceed $CT^{-1}h^{1-d}$ where so far $T\asymp1$.
\end{proposition}

\subsection {Singular homogeneous case}  
\label{sect-1-3}

\begin{theorem}\label{thm-1-6}
Let conditions $(\ref{0-9})$ and $(\ref{0-7})$ be fulfilled. Then 
$|I_m|\le Ch^{-d-(m-1)\kappa}$.
\end{theorem}

\begin{proof}
Let us replace $\Omega({\mathbf x},{\mathbf z})$ by $\Omega({\mathbf x},{\mathbf z}) 
\beta (z^1/\gamma_1)\cdots \beta (z^m/\gamma_m)$ where $\gamma_j\ge h$ and $\beta, {\bar\beta}$ are functions (on $\bR^d$) similar to $\chi,{\bar\chi}$ respectively. Then similarly to the analysis of the smooth case one can estimate the contribution of such partition element to $I_m$ by 
\begin{equation}
Ch^{-d}\bigl(\gamma_1\cdots \gamma_m)^{-1}(\gamma_1+\dots +\gamma_m)^{1-\kappa}
\label{1-13}
\end{equation}
and summation with respect to $\gamma_j\ge {\bar\gamma}=h$ results in the value of this expression as $\gamma_j={\bar\gamma}$ and the total estimate becomes what is claimed.

However one needs to consider the other partition elements when some of 
$\beta (z^j/\gamma_j)$ are replaced by ${\bar\beta} (z^j/{\bar\gamma})$. So we get ``sandwiches'' consisting of the factors 
\begin{equation*}
e(x^k,x^{k+1},\tau)\beta(z^{k+1}/\gamma_{k+1}) \cdots \beta(z^j/\gamma_j)e(x^j,x^{j+1},\tau)
\end{equation*}
with $j\ge k$ and in between them factors ${\bar\beta} (z^k/{\bar\gamma})$.

Let $J$ be the set of indices appearing in  ${\bar\beta} (z^k/{\bar\gamma})$ (for a given type of a ``sandwich''). One can see easily that the contribution of each ``sandwich'' to $I_m$ does not exceed
\begin{equation*}
Ch^{-dr}\prod_{j\notin J}\gamma_j^{-\kappa}\times \bigl(\int _{\{|z|\le {\bar\gamma}\}}|z|^{-\kappa}\,dz\bigr)^{r-1} \asymp Ch^{-dr}\prod_{j\notin J}\gamma_j^{-\kappa}\times {\bar\gamma}^{(d-\kappa)(r-1)}
\end{equation*}
where $r$ is the number of factors of each type.  Then after summation with respect to $\gamma_j\ge{\bar\gamma}$ one gets the same expression with $\gamma_j={\bar\gamma}$ i.e. $Ch^{-dr}{\bar\gamma}^{\kappa (m-r)+(d-\kappa)(r-1)}=Ch^{-dr}{\bar\gamma}^{-\kappa (m-1)+d(r-1)}$
which is exactly what we want as ${\bar\gamma}\asymp h$.
\end{proof} 

It immediately follows from the proof a stronger condition

\begin{proposition} \label{prop-1-7}
Let conditions $(\ref{0-9})$ and $(\ref{0-7})$ be fulfilled. Then replacing $\Omega({\mathbf x},{\mathbf z})$ by $\Omega({\mathbf x},{\mathbf z}) 
{\bar\beta} (z^1/\gamma)\cdots {\bar\beta} (z^m/\gamma)$ results in the error not exceeding
\begin{equation}
Ch^{-d-(m-2)\kappa}\gamma^{-\kappa}.
\label{1-14}
\end{equation}
\end{proposition}

Now let assume  instead of condition (\ref{1-4}) or (\ref{1-5}) that

\begin{claim}\label{1-15}
$a(x,\xi)$ is microhyperbolic on energy level $0$ and microhyperbolicity directions are (at each point) $\ell_\xi\cdot \partial_\xi$\,\footnote{\label{foot-3} So $\ell_x=0$.} with $\ell_\xi=\ell_\xi (x,\xi)$.
\end{claim}

\begin{proposition}\label{prop-1-8}
Let conditions $(\ref{0-9})$, $(\ref{0-7})$ and $(\ref{1-15})$ be fulfilled. Then  replacing $\Omega({\mathbf x},{\mathbf z})$ by $\Omega({\mathbf x},{\mathbf z}) 
{\bar\beta} (z^1/\gamma)\cdots {\bar\beta} (z^m/\gamma)$ results in the error not exceeding
\begin{equation}
Ch^{1-d-(m-2)\kappa}\gamma^{-1-\kappa}.
\label{1-16}
\end{equation}
This is equivalent to taking $T\asymp \gamma$ in $(\ref{1-8})$ and plugging Schwartz kernel of it instead of $e(x,y,0)$ in the definition of $I_m$.
\end{proposition}

\begin{proof} Proof follows from the combined arguments of the proofs of Theorem \ref{thm-1-6} and Proposition  \ref{prop-1-1}; in this case one needs to consider only ``sandwiches'' containing  at least one factor $\beta (x^j/\gamma_j)$  with $\gamma_j\ge \gamma$ which accounts for  a factor $h/\gamma_j$ and summation with respect to partition results in an extra factor $h/\gamma$.
\end{proof}

So one needs to study expression (\ref{1-12}) with some $T=T^*$; I remind that the remainder estimate contains factor $T^{*\,-1}$. One can decompose ${\bar\chi}_{T^*}(t)$ into the sum of ${\bar\chi}_{\bar T}(t)$ and $\chi_T(t)$ with $T$ running between ${\bar T}$ and $T^*$ and also one can take ${\bar T}=Ch$. Then expression (\ref{1-12}) becomes the sum of the similar expressions
with ${\bar\chi}_T(t)$ (with $T=T^*$) replaced by $\phi_{jT_j}(t)$
where \emph{either} $\phi_j=\chi$ and ${\bar T}\le T_j\le T^*$ \emph{or} $\phi_j={\bar\chi}$ and $T_j={\bar T}$.

In this expression as $\phi_j=\chi$ one can replace $\int_{-\infty}^0(\dots )\,d\tau $ by $(\dots)|_{\tau=0}$ simultaneously replacing $h^{-1}\chi_T(t)$ by $it^{-1}\chi_T(t)=T^{-1}\phi_T(t)$ with $\phi (t)= it^{-1}\chi (t)$; so we get a modified expression (\ref{1-12}) with $r$ factors  ${\bar\chi}_{\bar T}(t_j)$ and $\tau_j$ snapped to $0$ for $j\in J$, $r=\# J$ and integration over $\bR^{-\, (m-r)}$ and $(m-r)$ factors $\phi_T(t_k)$, $k\notin J$; furthermore, factor $h^{-m}$ is replaced by $h^{-r}\prod_{k\notin J}T_k^{-1}$.

\begin{proposition}\label{prop-1-9}
Let conditions $(\ref{0-9})$ and $(\ref{1-15})$ be fulfilled and let
$\omega $ be a smooth function, 
\begin{equation}
\omega = O \bigl( (|x^1-x^2|+\dots +|x^m-x^1|)^K\bigr).
\label{1-17}
\end{equation}
Then $I_m=O(h^{1-d})$ as $K>1$ and $I_m=O(h^{1-d}|\log h|)$ as $K=1$. 
\end{proposition}
\begin{proof}
Proof follows from the combined arguments of the proofs of Theorem \ref{thm-1-6} and Proposition  \ref{prop-1-1} like in Proposition \ref{prop-1-8}. Here however the main contribution (as $K\ge 1$) is delivered by zone $\{|x^1-x^2|+\dots +|x^m-x^1|\asymp 1\}$. 
\end{proof}
One can consider certain generalizations but I will do it later.

\section{Calculations}\label{sect-2}

Now our purpose is to go from implicit Tauberian expression (\ref{1-12}) to more explicit one.

\subsection{Constant Coefficients Case}
\label{sect-2-1} 

Let us first consider case $A(x,\xi)=A(\xi)$. In this case 
\begin{equation}
e(x,y,\tau)= (2\pi h)^{-d}\int e^{ih^{-1}\langle x-y,\xi\rangle}E(\xi)\,d\xi
\label{2-1}
\end{equation}
where $E(\xi,\tau)$ is the matrix projector corresponding to $A(\xi)$.
Then
\begin{multline}
I_m  = (2\pi h)^{-dm} \int \int \omega (x^1,\dots,x^m) E(\xi^1,0)\cdots E(\xi^m,0) \times\\
e^{ih^{-1}\bigl(\langle x^1-x^2,\xi^1\rangle+\langle x^2-x^3,\xi^2\rangle +\dots +\langle x^m-x^1,\xi^m\rangle\bigr)}\,dx^1\cdots dx^m\,d\xi^1\cdots d\xi^m.
\label{2-2}
\end{multline}\vglue-5truemm
\begin{claim}\label{2-3}
From now and until the end of the paper I am assuming that $m=2$.
\end{claim}
Without any loss of the generality one can assume that either $\omega (x,y)$ is of the form
\begin{equation}
\omega (x,y)=\Omega \bigl({\frac 1 2}(x+y), x-y\bigr).
\label{2-4}
\end{equation}
or it is of the  same singular type as before but multiplied by $(x_k-y_k)$. However in the latter case (under microhyperbolicity condition) one can apply a Tauberian approximation for $e(x,y,\tau)$ equal $0$ with the remainder estimate 
$O(h^{1-d}|x-y|^{-1})$ (in the same trace class as before) which leads to $I\approx 0$ with the sought remainder estimate $O(h^{1-d-\kappa })$.

In the former case (\ref{2-4})  we get
\begin{equation}
I \Def I_2 = \int \cJ (x)\,dx,\label{2-5}
\end{equation}
where
\begin{align}
\cJ(x)=
&2(2\pi h)^{-2d}   \iiint   \Omega (x, z) E(\xi,0) E(\eta,0) 
e^{ih^{-1}\langle z,\xi-\eta \rangle}\,dz d\xi d\eta =
G(x) h^{-d-\kappa},\label{2-6}\\
\intertext{with}
G (x) =& \iint  {\hat\Omega} (x, \xi-\eta) E(\xi,0) E(\eta,0) \,d\xi d\eta,\label{2-7}
\end{align}
and
\begin{equation}
{\hat\Omega} (x, \zeta) =2(2\pi )^{-2d}\int  \Omega (x, z) 
e^{i \langle z,\zeta \rangle}  \,dz.
\label{2-8}
\end{equation}
One always can take $\Omega$ having a compact support with respect to $x$ (since we had originally cutoffs $\psi_1(x^1),\dots, \psi_m(x^m)$.

\begin{remark}\label{rem-2-1} (i) One can easily generalize (\ref{2-5})--(\ref{2-8}) to $m>2$.

\medskip
(ii) Integral (\ref{2-8}) converges as $|z |\le 1$ since $\kappa<d$. On the other hand it defines a distribution with respect to $\zeta$ which is  positively homogeneous of degree $\kappa -d$ and also is smooth  as $\zeta\ne 0$; thus ${\hat\Omega }\in L^1_\loc$ and (\ref{2-7}) is well-defined. However generalization to $m>2$ is not that easy.
\end{remark}

\subsection{General Microhyperbolic Case}
\label{sect-2-2}
Note first that due to the microhyperbolicity condition (\ref{1-15}) one should take $T\asymp \gamma$ as $m=2$\,\footnote{\label{foot-4} And $T_j\asymp |x^j-x^{j+1}|$ in the general case.}. Otherwise as $T\in [ Ch^{1-\delta}, T^*]$, $T^*$ is the small constant, the contribution of $[T/2,T]\cup[-T,-T/2]$ would be negligible.

To calculate $u$ let us apply the successive approximation method on the time interval $[-T,T]$ with $h^{1-\delta}\le T$.  Then plugging the successive approximation into any copy on that interval we arrive to an error in $u$ in the trace norm equal to $O(h^{-d} (T^2/h)^n)$ where $n$ is the number of the first dropped term (starting from $0$). This leads to the error in $I$  $O\bigl(h^{-d-\kappa} (T^2/h)^n \gamma^{-\kappa}\bigr)$ as $T\ge \gamma$. Since under microhyperbolicity assumption (\ref{1-15}) we need to consider only $T\asymp \gamma$, the error is  $O(h^{-d} (T^2/h)^n T^{-\kappa})$. However if we just take $u=0$ then we get an error $O(h^{1-d}  T^{-1-\kappa})$.

Finding $T$ from the equation
\begin{equation*}
h^{-d} (T^2/h)^n  = h^{1-d}T^{-1}
\end{equation*}
we get 
\begin{equation}
T= h^{(n+1)/(2n+1)}
\label{2-9}
\end{equation}
(which is greater than $ h^{1-\delta}$ with $\delta>0$) and this leads to an error
\begin{equation}
O\bigl( h^{1-d- (n+1)(\kappa +1)/(2n+1)}\bigr)
\label{2-10}
\end{equation}

\begin{proposition}\label{prop-2-2}
Let conditions $(\ref{0-9})$, $(\ref{0-7})$ and $(\ref{1-15})$ be fulfilled. 
Then 

\medskip
\noindent
{\rm (i)} Using successive approximation as $|t|\le T$ given by $(\ref{2-2})$ and taking $u=0$ otherwise we get $I$ with an error given by $(\ref{2-10})$.

\medskip
\noindent
{\rm (ii)} In particular this is the sharp remainder estimate $O(h^{1-d-\kappa})$
as 
\begin{equation}
\kappa \ge (n+1)/n;
\label{2-11}
\end{equation}
in particular, as $\kappa \ge 2$ one can skip all perturbation terms and get the same answer $(\ref{2-4})-(\ref{2-7})$.
\end{proposition}

On the other hand, if we cannot skip some term, then this is given by the same formulae (\ref{2-4})--(\ref{2-7}) as before but with the factor $h^{-d-\kappa +s}$ instead of $h^{-d-\kappa}$ and with $\Omega$ replaced by $\Omega_s$ positively homogeneous of degree $-\kappa +s$ (provided these formulae have sense!). Then as long as $s<\kappa$ one can see that these terms are less than  the remainder estimate and we arrive to

\begin{proposition}\label{prop-2-3}
Let conditions $(\ref{0-9})$, $(\ref{0-7})$ and $(\ref{1-15})$ be fulfilled.
Then 

\medskip
\noindent
{\rm (i)} As $\kappa >1$ formulae $(\ref{2-4})-(\ref{2-7})$ provide an answer with the remainder estimate  $O(h^{1-d-\kappa})$.

\medskip
\noindent
{\rm (ii)} As $\kappa \le 1$ formulae $(\ref{2-4})-(\ref{2-7})$ provide an answer with the remainder estimate\newline $O(h^{{\frac 1 2}(1+\kappa)-d-\kappa -\delta})$ with arbitrarily small exponent $\delta>0$.
\end{proposition}

\subsection{Scalar Case}
\label{sect-2-3}

Let us completely analyze the case of scalar operator $A$.
\subsubsection{}\label{2-3-1} 

Assume first that $\omega=1$ and $\psi_1,\psi_2$ are smooth functions. Then one can rewrite  (\ref{1-9}) with $m=2$
\begin{equation}
h^{-2}\Tr\int _{(\tau_1,\tau_2)\in \bR^{-,2}}
F_{t_1\to  h^{-1}\tau_1,t_2\to h^{-1}\tau_2} \Bigl( {\bar\chi}_T(t_1) {\bar\chi}_T(t_2) 
\psi_1 U(t_1)\psi_2 U(t_2)  \Bigr)  \,d\tau
\label{2-12}
\end{equation}
with $T=T^*$ which is the largest value for which remainder estimate $O(T^{-1}h^{1-d})$ for the standard asymptotics was derived; here $T^*\asymp 1$.

If we replace some copies of ${\bar\chi}_T(t_k)$ by $\chi_{T_k}(t_k)$ with 
$C h \le T_k\le T^*$ then one can replace also operator $h^{-1}\int_{-\infty}^0 \bigl( \dots\bigr)\,d\tau_k$ by $T^{-1}\bigl( \dots\bigr)|_{\tau_k=0}$ and $\chi$ by $it^{-1}\chi $. 

If we do it with both $k=1,2$ then we get a term $O(h^{-d})$ (the better estimate is actually possible) and the summation with respect to all partitions with respect to $T_1,T_2$ results in $O(h^{-d}|\log h|^2)$ which differs from the proper estimate by $|\log h|^2$ factor. If we replace some copies of ${\bar\chi}_T(t_k)$ by ${\bar\chi}_h(t_k)$ then we do not make a transformation with respect to these factors but we gain a factor $h$ due to the size of the support. So after summation with respect to partition we arrive to estimate $O(h^{-d}|\log h|^{2-r})$ for $I$ where $r$ is the number of
${\bar\chi}_h(t_k)$ factors.

On the other hand expression (\ref{2-12}) is equal to
\begin{equation}
h^{-2}\Tr\int _{(\tau_1,\tau_2)\in \bR^{-,2}}
F_{t_1\to  h^{-1}\tau_1,t_2\to h^{-1}\tau_2}
\Bigl( {\bar\chi}_T(t_1) {\bar\chi}_T(t_2) 
\psi_1\psi_{2,t_1} U(t_1+t_2) ) \Bigr)  \,d\tau
\label{2-13}
\end{equation}
with $\psi_t =U(t)\psi  U(-t)$.

Applying standard approach we arrive to 
\begin{equation}
\cI \sim \sum_{n\ge 0}\varkappa_n h^{-d+n}
\label{2-14}
\end{equation}
where $\cI=\cI_2$ is defined by (\ref{1-12}).

Let us replace in (\ref{2-13}) $\psi_{2,t_1}$ by $\psi_2$. Plugging $t_{1,2}={\frac 1 2}t\pm z$, $\tau_{1,2}=\tau \pm \tau'$  we arrive to 
\begin{equation}
h^{-1}\Tr\int _\infty^0 \Bigl( \int_{\bR} 
\rho_T (t,\tau) \psi_1\psi_2 U(t) e^{-ih^{-1}t \tau }\,dt \Bigr) \,d\tau
\label{2-15}
\end{equation}
where $\rho_T(t,\tau)=\rho (t/T,\tau)$, $\tau<0$
\begin{equation}
\rho (t,\tau)= -\pi^{-1} h^{-1} \int_{\bR}
{\bar\chi}_T ({\frac 1 2}t+z) {\bar\chi}_T ({\frac 1 2}t-z)
z^{-1} \sin (h^{-1}Tz\tau ) \,dz 
\label{2-16}
\end{equation}
is $C_0^\infty ([-2,2])$ and one can prove easily that 
\begin{equation}
|\partial_t^n\bigl(\rho (t,\tau)\mp {\bar\chi}^2(t/2)\bigr)|\le C_{nm} (1+|\tau|Th^{-1})^{-m}\qquad \forall m,n\ \forall \tau \lessgtr 0.
\label{2-17}
\end{equation}

Then due to (\ref{2-17}) only zone $\{|\tau|\le h^{1-\delta}\}$ gives a non-negligible contribution to this error and due to the microhyperbolicity condition there $|\Tr  \psi_1\psi_2U(t)|\le Ch^{-d}(1+|t|h^{-1})^{-m}$ which together with (\ref{2-17}) implies 

\begin{claim}\label{2-18}
Under microhyperbolicity  condition (\ref{1-4}) expression (\ref{2-15}) is equal modulo $O(h^{1-d})$ to  the same expression with $\rho$ replaced by 
${\bar\chi} ^2(t/2)$.
\end{claim}

On the other hand, if we replace $\psi_{2,t_1}$ by $\psi_{2,t_1}-\psi_2=t_1\psi'_{2,t_1}$ then we can apply the same transformation as before just getting rid of one factor $h^{-1}$ and integration with respect to $\tau_1$, which simply snaps to $0$, resulting in expression, similar to (\ref{2-15}) but with $\rho \psi_2$ replaced by
\begin{equation}
\rho' (t,\tau, x)= (2\pi)^{-1} i\int_{\bR}
{\bar\chi}_T (z) {\bar\chi}_T (t-z)
e^{ih^{-1}\tau z}\psi'_{2,z} \,dz 
\label{2-19}
\end{equation}
which satisfies inequality similar to (\ref{2-17})
\begin{equation}
|\partial_t^n \rho (t,\tau)|\le C_{nm} (1+|\tau|Th^{-1})^{-m}\qquad \forall m,n\ \forall \tau \lessgtr 0.
\label{2-20}
\end{equation}
and therefore 
\begin{claim}\label{2-21}
Under microhyperbolicity condition (\ref{1-4}) this new (\ref{2-15})-type expression  is  $O(h^{1-d})$.
\end{claim}

So, we are left with expression (\ref{2-15}) with $\rho(t)={\bar\chi}^2(t/2)$ but due to the standard theory we get modulo $O(h^{1-d})$ expression
\begin{equation}
\Tr\psi_1\psi_2 E(0) \equiv  (2\pi h)^{-d}\iint_{\{a(x,\xi)<0\}}  \psi_1\psi_2\, dx\,d\xi.
\label{2-22}
\end{equation}
So, $\cI$ is given by (\ref{2-22}) modulo $O(h^{1-d})$ and therefore 
\begin{equation}
\varkappa_0= (2\pi)^{-d} \int \iint_{\{a(x,\xi)<0\}}  \psi_1\psi_2\, dx\,d\xi
\qquad \text{in\ (\ref{2-14})}.
\label{2-23}
\end{equation}

\subsubsection{}
Then in the general smooth case we get

\begin{proposition}\label{prop-2-4}  Let $\omega$ and $\psi_1,\dots, \psi_m$  be   smooth functions and let   $(\ref{0-9})$  and microhyperbolicity condition $(\ref{1-4})$ be fulfilled. Then  with an error $O(T^{-1}h^{1-d})$ where $T\asymp1$ here decomposition $(\ref{2-14})$ holds with 
\begin{equation}
\varkappa_0= (2\pi)^{-d}\iint_{\{a(x,\xi)<0\}}  \omega(x,x)\psi_1(x)\psi_2(x)\, dx\,d\xi.
\label{2-24}
\end{equation}
\end{proposition}

\begin{proof} Follows from the standard decomposition (\ref{1-10})-(\ref{1-11}).
\end{proof}

\subsubsection{} Consider now the case of singular homogeneous $\omega$. First, let us consider $\cI_\gamma$ defined by (\ref{1-12}) with $\omega=1$ and $\psi_1, \psi_2$ replaced by $\psi_{1,\gamma}, \psi_{2,\gamma}$ which are some smooth functions scaled at some point $z$ with the scaling parameter  $\gamma\in (h^{1-\delta},h^\delta)$. To have microhyperbolicity condition sustain scaling we replace it by (\ref{1-15}). Then (\ref{2-14}) implies
\begin{equation}
\cI'   \sim   \sum_{n, m\ge 0}  \varkappa_{nm} h^{-d+n}\gamma^{m-n+d}
\label{2-25}
\end{equation}
and obviously
\begin{equation}
(2\pi)^{-d}\iint_{\{a(x,\xi)<0\}}  \psi_{1,\gamma}(x)\psi_{2,\gamma}(x)\, dx\,d\xi
\sim \sum_{  m\ge 0}  \varkappa'_m \gamma^{m+d}. 
\label{2-26}
\end{equation}
One can see easily that in (\ref{2-25}) terms with $m=0$ would be the same for operator $A^0_z=a_0(z,hD)$ where  $a_0(x,\xi)$ is the principal symbol of $A$; this $z$ is not necessarily the original one, but distance between them should not exceed $c\gamma$; similarly  in (\ref{2-26}) term with $m=0$ coincides with the left-hand expression with $a(x,\xi)$ replaced by $a(z,\xi)$.

What is more, under condition (\ref{1-15}) integration with respect to $x$ is not needed, so all these results would hold (without factor $\gamma^d$ in the decomposition and estimates) without it; thus one can take $z=x$ (or $y$, does not matter).  

Thus we arrive to 

\begin{proposition}\label{prop-2-5} Let $\cI'$ be defined by $(\ref{1-12})$ with $\omega=1$ and $\psi_1, \psi_2$ replaced by $\psi_{1,\gamma}, \psi_{2,\gamma}$ which are the same smooth functions scaled at some point $z$ with the parameter 
$\gamma\in (h^{1-\delta},h^\delta)$. Let  $\cI^{0\prime}$ be defined the same way but with $U(t)$ replaced by  $U^0(t)=e^{ih^{-1}tA^0}$ where $A^0=a(z,hD)$ and later $z$ is set to $x$. Then 
$\cI'-\cI^{0\prime}_m=O( h^{1-d}\gamma^d)$
\end{proposition}
 
 Now we can calculate $\cI$ in the scalar case:

\begin{proposition}\label{prop-2-6}  In frames of proposition \ref{prop-2-5} as $\omega$ satisfies $(\ref{1-7})$ and $\kappa>0$ $\cI - \cI^0= O(h^{1-d- \kappa})$ where $\cI^0 $ is defined for constant-coefficient operator obtained by freezing coefficients of $A$ at point $x$ (or $y$, does not matter).
\end{proposition}

\begin{proof} Consider three zones: $\{|x-y|\gtrsim \gamma_1\}$ with $\gamma_1\asymp h^\delta$, $\{\gamma \lesssim  |x-y|\lesssim \gamma_1\}$ with $\gamma_0\asymp h^{1-\delta}$ and  $\{ |x-y|\lesssim \gamma\}$; then the contribution of the first zone to the reminder for $\cI$ and $\cI^0$ does not exceed $Ch^{1-d}\gamma_1^{-1-\kappa}=O(h^{1-d-\kappa})$ (while main parts are $0$); in virtue of proposition \ref{prop-2-5} and decomposition of subsection \ref{sect-1-2} the contribution of the second zone to $\cI-\cI^0$ does not exceed $O(h^{1-d}\gamma^{-\kappa})=o(h^{1-d-\kappa})$.

In the third zone one can apply the method of successive approximations resulting in 
\begin{equation*}
\cI-\cI^0\sim h^{-d}\sum_{m+n+k\ge 1} \varkappa''_{mnk}  h^{-d+n-m+k-\kappa}\gamma^{2m-n}.
\end{equation*}
However since the final answer does not depend on $\gamma$ only terms with $2m=n$ are posed to survive just resulting in 
$\bigl(\varkappa  +o(1)\bigr) h^{-d+1-\kappa}$.
\end{proof}

Summarizing results of section \ref{sect-1}, proposition \ref{prop-2-6} and formulae (\ref{2-5})--(\ref{2-8})  we arrive to

\begin{theorem}\label{thm-2-7} Let  $A$ be a scalar operator satisfying conditions $(\ref{0-9})$ and $(\ref{1-15})$. Then
\begin{equation}
I = \int \cJ (x)\psi_1(x)\psi_2(x)\,dx + O(h^{1-d-\kappa}),
\label{2-27}
\end{equation}
where
\begin{align}
\cJ(x)=
&2(2\pi h)^{-2d}  \iiint    E(x,\xi,0)  \Omega (x, z) E(x,\eta,0)
e^{ih^{-1}\langle z,\xi-\eta \rangle}\,dz d\xi d\eta =
\label{2-28}\\
&2(2\pi h)^{-2d}  \iiint _{\{a(x,\xi)<0,\ a(x,\eta)<0\} }  \Omega (x, z) 
e^{ih^{-1}\langle z,\xi-\eta \rangle}\,dz d\xi d\eta =
G(x) h^{-d-\kappa},\notag\\
\intertext{with}
G (x) =& \iint E(x,\xi,0)  {\hat\Omega} (x, \xi-\eta) E(x,\eta,0)\,d\xi d\eta=\label{2-29}\\
& \iint _{\{a(x,\xi)<0,\ a(x,\eta)<0\} }  {\hat\Omega} (x, \xi-\eta) \,d\xi d\eta,\notag
\end{align}
and ${\hat\Omega}$ is defined by $(\ref{2-8})$.
\end{theorem}

\begin{remark}\label{rem-2-8} (i) Alternatively one can prove this theorem using oscillatory integral representation of $u(x,y,t)$ as $|t|\le T=\epsilon$.

\medskip\noindent
(ii) Alternatively one can replace one or both copies of $x$ in $E(x,.,.)$ or in $a(x,0)$ by $y$.
\end{remark}

\begin{definition}\label{def-2-9}
We will refer to formulae (\ref{2-27})-(\ref{2-29}),(\ref{2-8}) as to \emph{standard Weyl expression} even in the matrix case. However in this case the third parts of (\ref{2-27}),(\ref{2-28}) should be skipped.
\end{definition}

\subsection{Schr\"odinger operator }
\label{sect-2-4}

Now my goal is to weaken and eventually to get rid off microhyperbolicity  condition for scalar operators. I start from the Schr\"odinger operator.

For a Schr\"odinger operator condition of microhyperbolicity (\ref{1-15}) means that 
\begin{equation}
V\ge  \epsilon_0.
\label{2-30}
\end{equation}
If this condition is violated let us introduce scaling functions $\rho(x)$, $\gamma (x)$ in the usual way  $\gamma =\epsilon  |V|$ and $\rho = \gamma^{1/2}$.

Then, the contribution of $B({\bar x},\gamma ({\bar x}))^2$  to the remainder does not exceed 
\begin{equation}
C(h/\rho \gamma)^{1-d-\kappa}\gamma ^{- \kappa}\asymp Ch^{1-d-\kappa}\rho^{d-1-\kappa}\gamma^{d-1}
\label{2-31}
\end{equation}
with $\rho =\rho({\bar x})$ and $\gamma=\gamma({\bar x})$ and then the contribution of zone 
\begin{equation}
\bigl\{(x,y):   |x-y|\le \epsilon \gamma(x )\bigr\}
\label{2-32}
\end{equation}
(where automatically $\gamma(x )\asymp \gamma(y)$)  to the remainder does not exceed
\begin{equation}
Ch^{1-d- \kappa}\int \rho^{d-1+ \kappa}\gamma^{-1}\,dx
\label{2-33}
\end{equation}
and with $\rho=\gamma^{1/2}$ here it becomes 
\begin{equation}
Ch^{1-d-\kappa}\int \gamma^{(d-3+\kappa)/2}\,dx;
\label{2-34}
\end{equation}
obviously, it is $O(h^{1-d-\kappa})$ provided \emph{either} $d+\kappa \ge 3$ \emph{or} 
\begin{equation}
|V|+|\nabla V|\ge  \epsilon_0.
\label{2-35}
\end{equation}
and $d+\kappa>1$ (which is surely the case). 

\begin{remark}\label{rem-2-10}
(i) Note that (\ref{2-35}) is microhyperbolicity condition (\ref{1-4}).

\medskip
\noindent
(ii) Actually one should take $\rho\gamma\ge Ch$ and thus to add $Ch^{1/3}$ and $Ch^{2/3}$ to $\rho$,$\gamma$ respectively (but it does not affect our conclusion due to the standard fact that if $\rho\gamma\asymp h$  then $h_\eff\asymp 1$ and condition (\ref{2-35}) is not needed. 
\end{remark}

Consider now the complement of zone (\ref{2-32}). Let us redefine there $\gamma(x)$ as $\gamma (x,y)={\frac 1 2} |x-y|$ and in this zone condition (\ref{2-35}) is not needed as one can see easily after rescaling $B(x,\gamma (x,y))$ to $B(0,1)$ due to proposition \ref{prop-1-5}. 

Therefore as $\gamma\ge \gamma(x)$ the contribution of $B(x,\gamma )^2\setminus \{\textsf{zone (\ref{2-32})}\}$ to the remainder does not exceed the same expression (\ref{2-31}) with $\rho=\gamma^{1/2}$. Then the contribution of the complement of zone (\ref{2-32}) to the remainder     does not exceed
\begin{equation}
Ch^{1-d-\kappa}\iint_{\{|x-y|\ge \epsilon\max(\gamma(x),\gamma(y))\}}
|x-y|^{(d-1+ \kappa)/2-1-d}\,dx\,dy.
\label{2-36}
\end{equation}
One can see easily that expression (\ref{2-36}) is $O(h^{1-d-\kappa})$ as $d+\kappa>3$ (so this case is already covered).

Further, expression (\ref{2-36}) does not exceed expression (\ref{2-34}) with $\gamma=\gamma(x)$  and expression
\begin{equation}
Ch^{1-d- \kappa}\int (|\log \gamma(x)|+1)\,dx
\label{2-37}
\end{equation}
as $d+\kappa<3$ and $d+\kappa=3$ respectively and both these expressions are $O(h^{1-d-\kappa})$ under condition (\ref{2-35}).

Again we get $O(h^{1-d-\kappa})$  provided either $d+\kappa > 3$ or condition (\ref{2-35}) is fulfilled. So, we arrive to

\begin{proposition}\label{prop-2-11}
Consider Schr\"odinger operator. Let either $d+\kappa>3$ or condition $(\ref{2-35})$ be fulfilled. Then the standard Weyl asymptotics $(\ref{2-27})-(\ref{2-29}),(\ref{2-8})$ holds with the remainder estimate  $O(h^{1-d-\kappa})$.
\end{proposition}

This completely covers the case $d\ge 3$. Furthermore, after proposition \ref{prop-2-11} is proven,  we can introduce scaling functions $\gamma=\rho =\epsilon  (|V|+|\nabla V|^2)^{1/2}+Ch^{1/2}$ and then applying the same arguments we  arrive to

\begin{proposition}\label{prop-2-12}
Consider Schr\"odinger operator. Let either $d+\kappa>2$ or condition 
\begin{equation}
|V|+|\nabla V|+|\nabla ^2V|\ge\epsilon
\label{2-38}
\end{equation}
be fulfilled. Then the standard Weyl asymptotics $(\ref{2-27})-(\ref{2-29}),(\ref{2-8})$ holds with the remainder estimate  $O(h^{1-d-\kappa})$.
\end{proposition}
This completely covers the case $d=2$.
As $d=1$ we get the required remainder estimate under condition (\ref{2-38}).

Now, combining this with the arguments of the proof of Theorem  4.4.9 of \cite{Ivr1} we get\footnote{\label{foot-5} I am leaving easy details to the reader; see also the proof of Theorem \ref{thm-2-19}.}

\begin{proposition}\label{prop-2-13}
Consider Schr\"odinger operator with $d=1$, $\kappa>0$. Then the standard Weyl asymptotics $(\ref{2-27})-(\ref{2-29}),(\ref{2-8})$ holds with the remainder estimate  $O(h^{1-d-\kappa})$.
\end{proposition}

\begin{remark}\label{rem-2-14} Actually all above results hold as $\kappa=0$ as well with the singular exception of $d=1$  when the remainder estimate $O(1)$ is recovered under condition
\begin{equation}
\sum_{|\beta|\le K} |\nabla_x^\beta V|\ge \epsilon;
\label{2-39}
\end{equation}
without it  the remainder estimate is $O(h^{-\delta})$ with arbitrarily small $\delta>0$.
\end{remark}

\subsection{Scalar Case. II}
\label{sect-2-5}
\subsubsection{}\label{sect-2-5-1}Let us consider general scalar operators.

\begin{remark}\label{rem-2-15} (i) Actually instead of condition (\ref{0-9}) one can make a cut-off with respect to $\xi$ replacing functions $\psi_j(x)$ by pseudo-differential operators $\psi_j(x,hD)$ with smooth compactly supported symbols; 

\medskip
\noindent
(ii) Alternatively we can replace $E(0)$ by $E(\tau,\tau')=E(\tau)-E(\tau')$ with conditions satisfied for $a-\tau$ and $a-\tau'$ instead of $a$.

\medskip
\noindent
(iii) Alternatively we can replace $E(0)$ by 
\begin{equation}
E'(\tau)=\int_{\bR} E(0,\tau')\varphi (\tau')\,d\tau'
\label{2-40}
\end{equation}
with smooth function $\varphi$ s.t. $\int_{\bR}\varphi (\tau')\,d\tau'=1$.

In all these cases obvious modifications of the  final formulae are needed.
\end{remark}

Now we can introduce scaling functions 
\begin{equation}
\gamma(x,\xi)=\epsilon \bigl(|\nabla_\xi a|^2+|a|\bigr) + Ch^{2/3},\qquad
\rho (x,\xi)=\gamma ^{1/2}(x,\xi)
\label{2-41}
\end{equation}
and repeat arguments of the previous subsection; then expression (\ref{2-33}) will be replaced by $Ch^{1-d-\kappa}M $ with
\begin{equation}
M=\int \rho ^{\kappa-1}\gamma^{-1}\,dx d\xi \asymp 
\int \bigl(|\nabla_\xi a|^2 +|a|\bigr)^{(\kappa-3)/2}\,dx d\xi  
\label{2-42}
\end{equation}
(in zone $\{\rho\gamma\ge Ch\}$).  Therefore we arrive to the remainder estimate $O(h^{1-d-\kappa})$ provided $M=O(1)$ as \emph{now integral in $M$ is taken over $B(0,1)$.}

This is definitely the case as $\kappa\ge 3$. Assume now that microhyperbolicity condition (\ref{1-4}) is fulfilled. Then $M=O(1)$ as $\kappa>1$; otherwise this condition becomes
\begin{equation}
\int_\Sigma |\nabla_\xi a|^{\kappa-1 }\,d\mu <\infty \quad \text{as } 0<\kappa<1, \qquad
\int_\Sigma |\log |\nabla_\xi a||\,d\mu<\infty
\label{2-43}
\end{equation}
with $\Sigma=\{a(x,\xi)=0\}$ and $d\mu =dx d\xi :da$ measure on $\Sigma$.

Thus we arrive to the following generalization of proposition \ref{prop-2-11}:

\begin{proposition}\label{prop-2-16} Let $A$ be a scalar operator satisfying condition $(\ref{0-9})$. Assume that the uniform version of condition\footnote{\label{foot-6} I.e.  $|a|+|\nabla_\xi a|\le \epsilon$ implies that $\Hess_{\xi\xi}a$ has $r$ eigenvalues which absolute values are greater than $\epsilon$.} 
\begin{phantomequation}
\label{2-44}
\end{phantomequation}
\begin{equation}
a=\nabla_\xi a =0 \implies \rank \Hess _{\xi\xi} a \ge r 
\tag*{$(\ref*{2-44})_r$}
\label{2-44-r}
\end{equation}
is fulfilled. Then

\medskip
\noindent
{\rm (i)} As $r+\kappa>3$   the standard Weyl asymptotics $(\ref{2-27})-(\ref{2-29}),(\ref{2-8})$ holds with the remainder estimate  $O(h^{1-d-\kappa})$;

\medskip
\noindent
{\rm (ii)} Under condition $(\ref{1-4})$ as $r+\kappa>1$   the standard Weyl asymptotics $(\ref{2-27})-(\ref{2-29}),(\ref{2-8})$ holds with the remainder estimate  $O(h^{1-d-\kappa})$.
\end{proposition}

\begin{proof}
In contrast to standard asymptotics we need to consider not points $(x,\xi)$ but pairs $(x,\xi;y,\eta)$ and the pure standard arguments work in zones 
\begin{equation}
\bigl\{ (x,\xi;y,\eta): |x-y|\le \epsilon \gamma(x,\xi), |\xi-\eta|\le \epsilon \rho (x,\xi)\bigr\}
\label{2-45}
\end{equation}
where also $\gamma (y,\eta)\asymp \gamma (x,\xi)$ and $\rho(y,\eta)\asymp \rho (x,\xi)$. Analysis in the complimentary zone I postpone until the proof of theorem \ref{thm-2-19} where it will be done in more general settings.
\end{proof}

Now introducing scaling functions 
\begin{equation}
\gamma(x,\xi)=\epsilon \bigl(|\nabla_{x,\xi} a|^2+|a|\bigr)^{1/2}+ Ch^{1/2},\qquad \rho (x,\xi)=\gamma (x,\xi)
\label{2-46}
\end{equation}
and repeating the same arguments we arrive to the following generalization of proposition \ref{prop-2-12}:

\begin{proposition}\label{prop-2-17} Let $A$ be a scalar operator satisfying condition $(\ref{0-9})$. Assume that the uniform version of condition \ref{2-44-r}  is fulfilled. Then as $r+\kappa>2$   the standard Weyl asymptotics $(\ref{2-27})-(\ref{2-29}),(\ref{2-8})$ holds with the remainder estimate  $O(h^{1-d-\kappa})$.
\end{proposition}

Again, combining this with the arguments of Theorem  4.4.9 of \cite{Ivr1} we arrive to the following generalization of proposition \ref{prop-2-12} 

\begin{proposition}\label{prop-2-18}
Let $A$ be a scalar operator satisfying conditions $(\ref{0-9})$ and $(\ref{2-44})_1$   and let $\kappa>0$. Then the standard Weyl asymptotics $(\ref{2-27})-(\ref{2-29}),(\ref{2-8})$ holds with the remainder estimate  $O(h^{1-d-\kappa})$.
\end{proposition}

\subsubsection{}\label{sect-2-5-2} Now  we   can prove our main result for scalar operators:

\begin{theorem}\label{thm-2-19} Consider scalar operator.  Let  conditions
$(\ref{0-9})$ and 
\begin{phantomequation}
\label{2-47}
\end{phantomequation}
\begin{equation}
\sum_{0\le k\le n}|\nabla _\xi^ka|\ge \epsilon_0
\tag*{$(\ref*{2-47})_n$}
\label{2-47-n}
\end{equation}
with some $n$ be fulfilled. Let $\omega$ satisfy $(\ref{1-7})$ and $\kappa>0$. Then the standard Weyl asymptotics $(\ref{2-27})-(\ref{2-29}),(\ref{2-8})$ holds with the remainder estimate  $O(h^{1-d-\kappa})$.
\end{theorem}

\begin{proof} [Proof Part I] In this part of the proof we consider at each step only zone (\ref{2-45}) where $\gamma$ will be defined in different ways later.  Treatment of the complementary zone will be described in Part II.

So, we proved the statement of the theorem under condition $(\ref{2-44})_1$ which is equivalent to $(\ref{2-47})_2$.

Let us apply induction with respect to $n$. \emph{Assume that under condition \ref{2-47-n} required estimate is proven\/}.

 In the general case (without condition \ref{2-47-n}) we can introduce scaling functions in the manner similar to (\ref{2-41}):
\begin{phantomequation}
\label{2-48}
\end{phantomequation}
\begin{equation}
\gamma(x,\xi)=\epsilon \bigl(\sum_{0\le k\le n}
|\nabla_\xi^ka|^{N/(n-k+1)} \bigr)^{(n+1)/N} + Ch^{(n+1)/(n+2) },\qquad
\rho (x,\xi)=\gamma ^{1/(n+1)}(x,\xi)
\tag*{$(\ref*{2-48})_n$}
\label{2-48-n}
\end{equation}
with   $N=(n+1)!$.  

Therefore under  assumption of induction we get again remainder estimate $Ch^{1-d+\kappa}M$ with $M$ given by (\ref{2-42}) where this time   the right-hand expression becomes
\begin{phantomequation}
\label{2-49}
\end{phantomequation}
\begin{equation}
M=\int \gamma ^{(\kappa -n-2)/(n+1)}\,dxd\xi;
\tag*{$(\ref*{2-49})_n$}
\label{2-49-n}
\end{equation}
under condition (\ref{1-4}) this expression becomes 
\begin{phantomequation}
\label{2-50}
\end{phantomequation}
\begin{equation}
M\asymp \int_\Sigma \gamma^{(\kappa-1)/(n+1)}\,d\mu\asymp \int_\Sigma \bigl(\sum_{1\le k\le n}
|\nabla_\xi^ka|^{1/(n-k+1)} \bigr)^{\kappa-1}\,d\mu 
\label{2-50-n}
\tag*{$(\ref*{2-50})_n$}
\end{equation}
which is $O(1)$ under assumption $|\nabla^{n+1}_\xi a|\ge \epsilon_0$ (as lower order derivatives with respect to $\xi$ are close to 0).  This is exactly condition $(\ref{2-47})_{n+1}$. 

\emph{So, now we have a proper estimate under condition 
$(\ref{2-47})_{n+1}$ instead of $(\ref{2-47})_n$ but now we also need condition $(\ref{1-4})$\/}.

Without condition (\ref{1-4}) we would need something different; f.e. ignoring integration with respect to $x$ one should assume that 
$\rank (\nabla_\xi ^{n+1}a) +\kappa > n+2$ where the rank of multilinear symmetric $m$-form $G$ is $d-\dim\Ker G$; $\Ker G=\{x: G(x, x^2,\dots, x^m)=0 \  \forall x^2,\dots,x^m\}$. This is rather unusable.

Instead I want to weaken condition (\ref{1-4}), replacing it by  
\begin{phantomequation}
\label{2-51}
\end{phantomequation}
\begin{equation}
\sum_{2\le j\le n+1,\ l:m + j:(n+1)\le 1 } |\nabla_x^l\nabla_\xi^j|\ge \epsilon_0
\tag*{$(\ref*{2-51})_{n+1,m}$}
\label{2-51-nm}
\end{equation}
for some $m>0$ which is not necessarily an integer. Obviously in our assumptions (\ref{1-4}) coincides with $(\ref{2-51})_{n+1,1}$.

Let us run a kind of nested induction. So, \emph{let us assume that under conditions $(\ref{2-47})_{n+1}$ and \ref{2-51-nm} remainder estimate $O(h^{1-d-\kappa})$ is proven.}

Now we can go to something similar (\ref{2-46}):
\begin{phantomequation}
\label{2-52}
\end{phantomequation}
\begin{multline}
\gamma(x,\xi)=\epsilon \bigl(\sum_{k,l: k:n +l:m \le 1}
|\nabla_\xi^k\nabla_x^l a|^{N s_{kl}} \bigr)^{1/N} + {\bar\gamma},\qquad {\bar\gamma}=Ch^{(n+1)/(m+n+2) },\\
\rho (x,\xi)=\gamma ^{(m+1)/(n+1)}(x,\xi),\qquad s_{kl}={\frac {n+1}{(m+1)(n+1)-(m+1)k-(n+1)l}}.
\tag*{$(\ref*{2-52})_{n,m}$}
\label{2-52-nm}
\end{multline}
Then we recover remainder estimate $Ch^{1-d-\kappa}M$ with $M$ defined by (\ref{2-42}) which is now
\begin{phantomequation}
\label{2-53}
\end{phantomequation}
\begin{equation}
M\asymp \int  \gamma^{-1 + (m+1)(\kappa-1)/(n+1)}\,dx d\xi\asymp
\int \rho   ^{-(n+1)/(m+1) +  (\kappa-1) }    \,dx d\xi.
\label{2-53-nm}
\tag*{$(\ref*{2-53})_{nm}$}
\end{equation}
Under condition $(\ref{2-47})_{n+1}$ we can assume without any loss of the generality that
\begin{equation}
a(x,\xi)= \sum_{0\le j \le n+1} b_j (x,\xi')\xi_1^{n+1-j},\qquad b_0=1,\quad b_1=0;
\label{2-54}
\end{equation}
we can always reach it by change of coordinates and multiplication of $A$ by an appropriate positive pseudo-differential factor. Then
\begin{equation}
\rho \asymp |\xi_1|+{\tilde\rho}, \qquad {\tilde\rho}={\tilde\gamma} (x,\xi')^{(m+1)/(n+1)},\qquad
{\tilde\gamma }=  \sum_{j, k,l: (k+j ):n +(l:m)\le 1}
|\nabla_{\xi'}^k\nabla_x^l b_j|^{ s_{(k+j)l}}  + {\bar\gamma }.
\label{2-55}
\end{equation}
Then 
\begin{phantomequation}
\label{2-56}
\end{phantomequation}
\begin{equation}
M\asymp
\int {\tilde\rho}   ^{-(n+1)/(m+1) +  \kappa }    \,dx d\xi'\asymp \int  {\tilde\gamma}^{-1 + (m+1)\kappa /(n+1)}\,dx d\xi'
\label{2-56-nm}
\tag*{$(\ref*{2-56})_{nm}$}
\end{equation}
(with an extra logarithmic factor as the power is 0). Then $M=O(1)$ as \begin{equation}
(m+1)\kappa /(n+1)>1.
\label{2-57}
\end{equation}

Moreover, $M=O(1)$ provided
there exists $(j,k,l)$ with 
$|\nabla_{\xi'}^k\nabla_x^l b_j|\ge \epsilon_0$ and \emph{either\/} 
$k \ge 1$, $(k+j -1):n+l:m \le 1$, $s_{k+j-1,l}<1$ 
\emph{or} $l\ge 1$, $(k+j):n+(l-1) :m\le 1$, $s_{k+j,l-1}<1$.

Therefore one can derive easily
\begin{claim}\label{2-58}
If remainder estimate $O(h^{1-d-\kappa})$ holds under condition 
$(\ref{2-51})_{n+1,m'}$ for every $m'<m$, then it also holds under condition
\ref{2-51-nm}.
\end{claim}
On the other hand there exists a discrete set  $\{m_\nu\}_{\nu=1,2,\dots}$  with $m_1<m_2<\dots $ such that if condition \ref{2-51-nm} is fulfilled for $m=m_\nu$ then it is fulfilled for all $m\in (m_\nu,m_{\nu+1})$ as well. 

This justifies induction with respect to $m$ running this set and therefore remainder estimate $O(h^{1-d-\kappa})$ holds under condition \ref{2-51-nm} no matter how large $m$ is. However if $m$ is large enough, condition (\ref{2-57}) is fulfilled and we do not need condition \ref{2-51-nm} anymore.

This concludes induction with respect to $n$. 
\end{proof}

\begin{proof} [Proof Part II] However in contrast to standard asymptotics we need to consider not points $(x,\xi)$ but pairs $(x,\xi;y,\eta)$ and the pure standard arguments work in zone (\ref{2-45}).

It follows from the standard theory that if $Qx$ and $Q_y$ have symbols supported in $\epsilon (\rho_x,\gamma_x)$- and $\epsilon (\rho_y,\gamma_y)$- vicinities of $(x,\xi)$ and $(y,\eta)$ respectively then
\begin{equation}
\|Q_x E Q_y\|_1 \le Ch^{-d}(\rho_x\gamma_x)^{d/2}(\rho_y\gamma_y)^{d/2}
\label{2-59}
\end{equation}
and moreover, if either $|x-y|\ge \epsilon_0 \gamma_x$ or $|\xi-\eta|\ge \epsilon_0 \rho_x$ then 
\begin{equation}
\|Q_x E Q_y\|_1 \le Ch^{1-d}(\rho_x\gamma_x)^{d/2-1}(\rho_y\gamma_y)^{d/2}.
\label{2-60}
\end{equation}
Surely the same will be true with $(x,\xi)$ and $(y,\eta)$ permuted. 

Then contribution of such pair to the error estimate does not exceed 
\begin{equation}
Ch^{1-d}(\rho_x\gamma_x)^{d/2-1}(\rho_y\gamma_y)^{d/2}|x-y|^{-\kappa}
\label{2-61}
\end{equation}
if  $|x-y|\ge \epsilon_0 \gamma_x$. 

Otherwise contribution of the pair $\psi_x Q_x$ and $Q_y$ to the error estimate does not exceed
$Ch^{1-d}(\rho_x\gamma_x)^{d/2-1}(\rho_y\gamma_y)^{d/2}\gamma^{-\kappa}$
where $\psi_1$, $(1-\psi_x)$ are supported in $\{|x-y|\ge \gamma\}$ and
$\{|x-y|\le 2\gamma\}$ and $\gamma \ge h \rho_x^{-1}$. 
 
 Furthermore, since 
 \begin{equation}
 |Q_xEQ_y|\le Ch^{1-d}\rho_x^{d/2-1} \gamma_x^{-1} \rho_y  ^{d/2} 
 \label{2-62}
 \end{equation}
 due to the standard arguments, contribution of the pair $(I-\psi_x)Q_x$ and $Q_y$ to the error does not exceed 
$Ch^{2-2d}\rho_x^{d-2}  \rho_y  ^{d} \gamma_x^{-2}\gamma_y^d\gamma^{d-\kappa}$
 Plugging $\gamma=h\rho_x^{-1}$ we estimate the contribution of the pair
 $Q_x$, $Q_y$ by
 \begin{equation}
 Ch^{1-d-\kappa}(\rho_x\gamma_x)^{d/2-1}(\rho_y\gamma_y)^{d/2} \rho_x^\kappa + 
 Ch^{2-d-\kappa }\rho_x^{-2+\kappa}  \gamma_x^{-2}\rho_y  ^{d} \gamma_y^d
 \label{2-63}
 \end{equation}
 which is larger than (\ref{2-61}). 
 
In these estimates we do not need non-degeneracy condition and therefore as $(y,\eta)$ and $(x,\xi)$ are given we can take 
\begin{equation}
\rho_x = \rho_y=  |x-y|^\sigma +|\xi-\eta| , 
\gamma_x = \gamma_y   |x-y|+|\xi-\eta|^{1/\sigma },
\label{2-64}
\end{equation}
where $\rho=\gamma^\sigma$ on the corresponding step of our analysis. 
Then as $(z,\zeta)$ are fixed contribution of 
$\{|x-z|\le \gamma, |y-z|\le \gamma, |\xi -\zeta |\le \rho, |\eta-\zeta|\le \rho, |x-y|+|\xi-\eta|^{1/\sigma}\ge \epsilon \gamma\}$
to the error does not exceed this expression 
\begin{equation}
 Ch^{1-d-\kappa} (\rho \gamma)^{d-1} \rho^\kappa  + 
 Ch^{2-d-\kappa }(\rho \gamma)^{d-2} \rho^\kappa
 \label{2-65}
\end{equation}
where the second term is less than the first one. 

Then the total contribution of the zone in question to the error does not exceed 
\begin{equation}
Ch^{1-d}\iiint  \gamma^{\sigma\kappa-\sigma -1}\,dy d\eta \, \gamma^{-1}d\gamma
\label{2-66}
\end{equation}
where equation is taken over $\{\gamma \ge \gamma_x\}$ and the integral in question is   equivalent  to $Mh^{1-d}$ where $M=1$  as $\sigma (\kappa -1)> 1$, 
\begin{equation}
M= \iiint  |\log \gamma (y,\eta) |\, dy d\eta 
\label{2-67}
\end{equation}
as $\sigma (\kappa-1)=1$ and due to \ref{2-47-n} $M\asymp 1$ as well,
\begin{equation}
M=\iiint  \gamma (y,\eta) ^{\sigma\kappa-\sigma-1}\,dy d\eta 
\label{2-68}
\end{equation}
as $\sigma (\kappa -1)<1$,  
 and on each step of the induction we already proved that $M\asymp 1$.
\end{proof}

\subsection{General Microhyperbolic Case. II}
\label{sect-2-6}

 Let us consider matrix operator. Let $\lambda_j(x,\xi)$ be eigenvalues of its principal part. Then 
$|\nabla _{x,\xi}\lambda_j|\le c$ and microhyperbolicity with respect to $\ell$  means that 
\begin{equation}
|\lambda_j(x,\xi)|\le \epsilon_0\implies (\ell \lambda_j) (x,\xi)\ge \epsilon_0\qquad \forall j.
\label{2-69}
\end{equation}
Let us consider zone
\begin{equation}
\cU_j =\bigl\{(x,\xi): |\lambda _j|\lesssim  \min_{k\ne j} |\lambda_k|\bigr\}
\label{2-70}
\end{equation}
and let us define here 
\begin{equation}
\gamma  \Def \min_{k\ne j} |\lambda_k|+{\frac 1 2}{\bar\gamma}
\label{2-71}
\end{equation}
and $\rho=\gamma$.
Consider zone
\begin{equation}
\bigl\{ \gamma \ge |x-y| +|\xi -\eta| + {\bar\gamma}\bigr\}
\label{2-72}
\end{equation}
and let us rescale  $x\mapsto x/\gamma$, $\xi\mapsto \xi/\gamma$, $\lambda_k\mapsto \lambda_k/\gamma$, $h\mapsto h/\gamma^2$ preserving microhyperbolicity condition (\ref{1-15}) and simultaneously making operator with $|\lambda_k|\ge 1$ for $k\ne j$ and therefore analysis of this operator is not different from the scalar one. Unfortunately we cannot use non-degeneracy conditions of subsections \ref{sect-2-4}--\ref{sect-2-5} which would not  survive this, but microhyperbolicity condition survives and \emph{we assume that $(\ref{1-15})$ is fulfilled\/}. 

Then as the main part of the asymptotics is given by the standard Weyl expression $(\ref{2-27})-(\ref{2-29})$, the contribution of zone (\ref{2-72}) (intersected with $\{\gamma \ge C_0{\bar\gamma}\}$) to the remainder 
 does not exceed

\begin{equation}
R_j=\int_{\Sigma_j \cap \{ \gamma \ge C{\bar\gamma} \}   } C\bigl(h\gamma^{-2}\bigr)^{1-d-\kappa}\gamma ^{-\kappa-2d}\,
d\wp_j\asymp 
Ch^{1-d-\kappa}\int_{\Sigma_j \cap \{ \gamma \ge {\bar\gamma} \}   }  \gamma ^{-2+\kappa}\,d\wp_j
\label{2-73}
\end{equation}
with $\Sigma_j=\{(x,\xi): \lambda_j=0\}$ and $d\wp_i=dxd\xi :d\lambda_j$ density on it. 

Let us fix ${\bar\gamma}=Ch^{1/2}$. Then in the complementary zone 
$\cup _{k\ne j}\{|\lambda_j|+|\lambda_k|\le C{\bar\gamma}\}$ 
one needs just to make a rescaling $x\mapsto x/{\bar\gamma}$, $\xi\mapsto \xi/{\bar\gamma}$ which sends $h$ to 1 and no microhyperbolicity condition would be needed and the contribution of this zone would not exceed 
\begin{equation}
R'_{jk}=Ch^{-d-\kappa/2}  \mes  \bigl \{|\lambda_j|+|\lambda_k|\le Ch^{1/2}\bigr\}.
\label{2-74}
\end{equation}

So, the total contribution of zone $\cup_j \cU_j$  to the remainder is given by $\sum_j R_j + \sum_{j,k:j\ne k}R'_{jk}$.

Assuming that
\begin{equation}
\wp_j \bigl(\Sigma_j: |\lambda_k|\le t\bigr)+ t^{-1}\mes \{|\lambda_j|+|\lambda_k|\le t\} =O(t^r).
\label{2-75}
\end{equation}
we get under additional assumption $r+\kappa > 2$ (which is always fulfilled as $r\ge 2$) 
that  $R_j=O(h^{1-d-\kappa})$ while $R'_{jk}=O(h^q)$ with
\begin{equation}
q=-d -{\frac 1 2}\kappa +{\frac r 2}.
\label{2-76}
\end{equation}
which is  $O(h^{1-d-\kappa})$ as well.

On the other hand, as $r+\kappa < 2$
we get  that $R_j= O(h^q)$, $R'_{jk}=O(h^q)$ with  $q$ given by (\ref{2-77}).

Finally, as $r+\kappa =2$ we get $R_j=O(h^{1-d-\kappa}|\log h|)$     and  $R_{jk}=O(h^{1-d-\kappa})$.

Assume temporarily that no more than two eigenvalues can be close to 0 simultaneously.  Then we are already done since in the zone complimentary to (\ref{2-72}) we redefine $\gamma = \epsilon (|x-y|+|\xi -\eta|)$ and apply the same rescaling as before and one does not need microhyperbolicity condition.

Let us apply induction by $m$ assuming that no more than $m$ eigenvalues can be close to 0 simultaneously. Then we can define on each step
\begin{equation}
\gamma (x,\xi) = \epsilon \max _{J: \#J=m} \min_{k\not\in J} |\lambda_k(x,\xi)| + {\bar\gamma}
\label{2-77}
\end{equation}
and repeat all above arguments. We will arrive to 

\begin{theorem}\label{thm-2-20}
Let conditions $(\ref{0-9})$, $(\ref{0-7})$, $(\ref{1-15})$ and $(\ref{2-75})$ be fulfilled.Then the standard Weyl asymptotics $(\ref{2-27})-(\ref{2-29})$ holds   with the remainder estimate

\medskip
\noindent
{\rm (i)} which is  $O(h^{1-d-(m-1)\kappa})$ as $r+\kappa>2$;

\medskip
\noindent
{\rm (ii)} which is $O(h^q)$ with $q$ defined by $(\ref{2-76})$ as $r+q\kappa <2$  and $O(h^{1-d-\kappa}|\log h|)$ as $r+\kappa =2$.
\end{theorem}

\begin{remark}\label{rem-2-21} Condition (\ref{2-75}) is fulfilled provided 
$\Lambda _{jk}=\{\lambda_j=\lambda_k=0\}$ are smooth manifolds of codimension
$r$ and $|\lambda_j| \asymp |\lambda_k|\asymp \dist ((x,\xi),\Lambda_{jk})$ in its vicinity; this assumption should be fulfilled  for all $j\ne k$.
\end{remark}

\bibliographystyle{alpha}

\begin{thebibliography}{BrIvr}


\bibitem[BrIvr]{IRO1}
{\sc M.~Bronstein, V.~Ivrii}.
\emph{Sharp Spectral Asymptotics for Operators with Irregular
Coefficients. Pushing the Limits},
Comm. Partial Differential Equations,
\textbf{28} (2003) 1\&2, 99--123.
 




%
%




\bibitem[Ivr1]{Ivr1}
{\sc V.~Ivrii}.
 \emph{Microlocal Analysis and Precise Spectral Asymptotics, 
 Springer-Verlag}, SMM, 1998, xv+731.


\bibitem [Ivr2] {MQT1} {\sc V.Ivrii}. Asymptotics of the ground state energy of
heavy molecules in the strong magnetic field. I.
\emph{Russian J. Math. Physics\/}, 4:1 (1996), 29-74.

\bibitem [Ivr3] {MQT2} {\sc V.Ivrii}. Asymptotics of the ground state energy of
heavy molecules in the strong magnetic field. II.
\emph{Russian J. Math. Physics\/},  5:3 (1997), 321-354

\bibitem [Ivr4] {MQT3}
{\sc V.Ivrii}. Heavy molecules in the strong magnetic field.
{\sl Russian J. Math. Physics\/}, {\bf 4} (1996), no 4, 449--456.



\bibitem [Ivr5] {IRO2} {\sc V.Ivrii}. Sharp Spectral Asymptotics for Operators with Irregular Coefficients. Pushing the Limits. II
\emph{Russian J. Math. Physics\/}, v. 28, no 1\&2, pp. 125Ð156, (2003). 

\bibitem [Ivr6] {EE} {\sc V.Ivrii}. Semiclassical asymptotics for exchange energy, S\'eminaire sur les \'Equations aux D\'eriv\'ees Partielles, 1993--1994, Exp. No. XX, 12 pp., ƒcole Polytech., Palaiseau, 1994. 


\bibitem [IS] {IS} {\sc V.Ivrii, I.M.Sigal}. Asymptotics of the ground state
energies of large Coulomb systems. {\sl Ann. Math\/} {\bf 138}
(1993), 243--335.
 \end{thebibliography}

\providecommand{\bysame}{\leavevmode\hbox to3em{\hrulefill}\thinspace}

\vglue .1truein
\begin{tabular}{rrl}
&{\hskip 220 pt} &Department of Mathematics,\cr
&&University of Toronto,\cr
&&40 St.George Str.,\cr
&&Toronto, Ontario M5S 2E4\cr
&&Canada\cr
&&ivrii@math.toronto.edu\cr
&&Fax: (416)978-4107\cr
\end{tabular}

\end{document}